\documentclass[10pt,a4paper]{article}
\usepackage[margin=3.24cm]{geometry}
\usepackage{amsfonts,amsmath,bm,dsfont,mathrsfs}
\usepackage{amssymb,amsthm}
\usepackage{enumitem,cases}
\usepackage{arydshln}
\usepackage[dvipsnames]{xcolor}
\colorlet{xlinkcolor}{red!50!black}
\usepackage{hyperref}[6.83]
\hypersetup{
	colorlinks = true,
	allcolors = NavyBlue,
	urlcolor = BrickRed,
}

\newtheorem{definition}{Definition}

\newtheorem{theorem}{Theorem}
\newtheorem{proposition}{Proposition}

\newtheorem{remark}{Remark}
\numberwithin{equation}{section}
\numberwithin{lemma}{section}
\numberwithin{definition}{section}
\numberwithin{assumption}{section}
\numberwithin{theorem}{section}
\numberwithin{proposition}{section}
\numberwithin{corollary}{section}
\numberwithin{remark}{section}

\setlist[description]{style=multiline,leftmargin=2em}
\setlist[itemize]{style=standard,leftmargin=2em}
\setlist[enumerate]{style=standard,leftmargin=2.2em,itemsep=3pt}

\title{\Large\bf
The robust isolated calmness of spectral norm regularized convex matrix optimization problems\thanks{This work was funded by
the National Natural Science Foundation of China (No. 12101101),
the China Postdoctoral Science Foundation (No. 2022M720631),
the Fundamental Research Funds for the Central Universities (Nos. 3132023204 and 3132024194).}}
\author{
Ziran Yin\thanks{School of Science, Dalian Maritime University, Dalian, Liaoning, 116026, China (\url{yinziran@dlmu.edu.cn}, \url{cxy11238556969@163.com}).}
\and
Xiaoyu Chen\footnotemark[2]
\and
Jihong Zhang\thanks{School of Science, Shenyang Ligong University, Shenyang, Liaoning, 110159, China (\url{zjh7815040x@163.com}).}
}

\date{}
\begin{document}
\maketitle

\begin{abstract}
	
This paper aims to provide a series of characterizations of the robust isolated calmness of the Karush-Kuhn-Tucker (KKT) mapping for spectral norm regularized convex optimization problems. By establishing the variational properties of the spectral norm function, we directly prove that the KKT mapping is isolated calm if and only if the strict Robinson constraint qualification (SRCQ) and the second-order sufficient condition (SOSC) hold. Furthermore, we obtain the crucial result that the SRCQ for the primal/dual problem and the SOSC for the dual/primal problem are equivalent. The obtained results can derive more equivalent conditions of the robust isolated calmness of the KKT mapping, thereby enriching the stability theory of spectral norm regularized optimization problems and enhancing the usability of isolated calmness in algorithm applications.

	
\medskip
\noindent
{\bf MSCcodes:}
90C25, 90C31, 65K10
\end{abstract}

\section{Introduction}
Consider the spectral norm regularized matrix optimization problem\\
\begin{equation}\label{p}
	\begin{array}{cl}
		\min\limits_{X\in \Re^{m\times n}} & h(\mathcal Q X)+\langle C,X\rangle+\|X\|_2\\
		{\rm s.t.} & \mathcal BX-b\in \mathcal P,
	\end{array}
\end{equation}
where $h:\Re^d\rightarrow \Re$ is twice continuously differentiable convex function and is essentially strictly convex,
$\mathcal Q: \Re^{m\times n}\rightarrow \Re^d$ and $\mathcal B:\Re^{m\times n}\rightarrow \Re^l$ are two linear operators, $C\in\Re^{m\times n}$ and $b\in\Re^l$ are given data and $\mathcal P\subseteq \Re^l$ is a given convex polyhedra cone. The function $\|X\|_2$ is the spectral norm function, namely, the largest singular value of $X$.
Without loss of generality, we suppose that $m\leq n$ in what follows.
For convenience, in the following text, we always use $\theta$ to represent the spectral norm function, i.e., $\theta(\cdot)=\|\cdot\|_2$.
Problem (\ref{p}) has a wide range of applications in various fields, such as matrix approximation problems \cite{GW1993}, matrix chebyshev polynomial problems \cite{TT1998} and $H_{\infty}$ synthesis problems \cite{AD2006,VDA2007,Yuan2024}, etc. Recently, spectral norm regularized problem has also been applied in deep learning and neural networks problems \cite{ANCP2023,YM2017,GYS2024}.

Stability analysis theory is very crucial in studying the convergence of algorithms. As one of the important concepts in stability properties, isolate calmness (see Definition \ref{def-isocalm}) may guarantee the linear convergence rate of some algorithms, such as the alternating direction method of multipliers \cite{HSZ2018} and the proximal augmented Lagrangian method \cite{Roc1976}. There are many publications on studying the  isolated calmness of the KKT mapping (ICKKTM) for optimization problems. For the nonlinear semidefinite programming, Zhang and Zhang \cite{ZZ2016} show that the ICKKTM can be derived from the SOSC and the SRCQ. Zhang et al. \cite{ZZW2017} demonstrate that the SOSC and the SRCQ are both sufficient and necessary for the ICKKTM for the nonlinear second order cone programming problem. It is important to note that Ding, Sun and Zhang \cite{DSZ2017} study the robust isolated calmness of a large class of cone programming problems. They prove that if and only if both the SOSC and the SRCQ hold, the KKT mapping is robustly isolated calm.

When the optimization problem has some kind of special linear structure, by establishing the equivalence between the constraint qualification of the primal (dual) problem and second order optimality condition of the dual (primal) problem, one can obtain more characterizations of stability properties. For instance, for the standard semidefinite programming problem, Chan and Sun \cite{CS2008} show that the constraint nondegeneracy for the primal (dual) problem is equivalent to the strong SOSC for the dual (primal) , then they obtain a series of equivalent characterizations of the strong regularity of the KKT point.
Han, Sun and Zhang \cite{HSZ2015ar} discover the equivalence relationship between the dual (primal) SRCQ and the primal (dual) SOSC for convex composite quadratic semidefinite programming problem, thereby deriving a series of equivalent conditions of ICKKTM.
For the case where $\theta$ in problem (\ref{p}) is the nuclear norm function, Cui and Sun \cite{CS2018} show that the primal (dual) SRCQ is equivalent to the dual (primal) SOSC. Therefore, they derive more equivalent conditions of the robust ICKKTM.

Inspired by the work in \cite{CS2018}, given the widespread application of spectral norm regularized convex matrix optimization problems, a natural question is whether the results in \cite{CS2018} can be extended to spectral norm regularized convex optimization problems. We provide an positive answer in this paper. Due to the special structure of the critical cone of spectral norm function, we provide several equivalent conditions of the robust ICKKTM. That is to say, the conclusions in \cite{CS2018} are still valid for problem (\ref{p}).
Compared to \cite{CS2018}, the difference in this paper is that we use the established variational properties of the proximal mapping of spectral norm to provide a direct proof of the ICKKTM for problem (\ref{p}), while \cite{CS2018} provides an indirect proof of the isolated calmness based on \cite[Theorem 24]{DSZ2017} regarding the optimization problems with a $\mathcal C^2$-cone reducible constraint set. In addition, compared to the SOSC, the SRCQ for primal or dual problem is easier to verify, which can enhance the practicality of isolated calmness in algorithm research.

The organization of the subsequent content is as follows. We provide some notations and preliminaries on variational analysis which will be used in the following text in Section 2. The variational properties of the spectral norm function are studied in Section 3, including the relationship between the critical cones of the spectral norm function and its conjugate function, and the explicit expression of the directional derivative of the proximal mapping of the spectral norm function. The results in Section 3 play a vital role in obtaining the important conclusions of this paper. In Section 4, we prove that the primal(dual) SRCQ holds if and only if the dual(primal) SOSC holds and thus establish more equivalent conditions for the robust ICKKTM for problem (\ref{p}). We conclusion this paper in Section 5.

Some common symbols and notations for matrices are as follows:

\begin{itemize}
	\item For any positive integer $t$, we denote by $[t]$ the index set $\{1,\cdots, t\}$.
	For any $Z\in \Re^{m\times n}$,  the $(i,j)$-th entry of $Z$ is denoted as $Z_{ij}$, where $i\in [m]$, $j\in[n]$.
	Let $\mu\subseteq[m]$ and $\nu\subseteq[n]$ be two index sets. We write $Z_{\nu}$ to be the sub-matrix of $Z$ that only retains all columns in $\nu$, and $Z_{\mu\nu}$ to be the $|\mu|\times|\nu|$ sub-matrix of $Z$ that only retains all rows in $\mu$ and columns in $\nu$.
	\item For any $d\in \Re^m$, ${\rm Diag}(d)$ represents the $m\times m$ diagonal matrix, where the $i$-th diagonal element is $d_i$, $i\in[m]$.
	\item Using ``trace" to represent the sum of diagonal elements in a given square matrix. For any two matrices $P$ and $Q$ in $\Re^{m\times n}$, the inner product of $P$ and $Q$ is written as  $\langle P,Q\rangle:={\rm trace}(P^TQ)$.
	The Hadamard product of matrices $P$ and $Q$ is represented by the symbol ``$\circ$", i.e., the $(i, j)$-th entry of $P\circ Q\in \Re^{m\times n}$ is $ P_{ij}Q_{ij}$.
	\item Let $\mathcal S^w$ be the linear space of all $w\times w$ real symmetric matrices, $\mathcal S_+^w$ and $\mathcal S_-^w$ be the cones of all $w\times w$ positive and negative semidefinite matrices, respectively.
\end{itemize}
\section{Notation and Preliminaries}
Let $\mathcal{X}$ and $\mathcal{Y}$ be two finite dimensional real Euclidean spaces. Given $\rho>0$, define the ball $\mathbb B_\rho(z):=\{z'\in \mathcal X~|~\|z'-z\|\leq \rho\}$, $\forall z\in \mathcal X$.
Let $D\subseteq \mathcal X$ be a non-empty closed convex set.
For any $z\in D$, the tangent cone (see \cite[Definition 6.1]{Row98}) of $D$ at $z$ is defined by ${\mathcal T}_D(z):=\{d\in \mathcal X~|~ \exists z^k\rightarrow z$ with $z^k\in D$ and  $\tau^k\downarrow 0$ s.t. $(z^k-z)/\tau^k\rightarrow d\}$.
We use $\mathcal N_D(x):=\{d\in \mathcal X:\langle d,z-x\rangle\leq 0~\forall z\in D\}$ to denote the normal cone of $D$ at $x\in D$.
Denote by $\delta_D(x)$ the indicator function over $D$, i.e., $\delta_D(x)=0$ if $x\in D$, and $\delta_D(x)=\infty$ if $x\notin D$.
Define the support function of the set $D$ as $\sigma(y, D):=\sup_{x\in D}\langle x,y\rangle, y\in \mathcal X$.
For a given $x\in \mathcal X$, define $\Pi_D(x):=\arg \min\{\|d-x\|~|~d\in D\}$ as the projection mapping onto $D$.
Suppose the function $f:\mathcal{X}\rightarrow (-\infty,+\infty]$ is proper closed convex, denote by ${\rm dom} f:=\{x\mid f(x)<\infty\}$ the effective domain of $f$, by $f^*$ the conjugate of $f$ and by $\partial f$ the subdifferential of $f$. For more details, one can refer standard convex analysis \cite{Row98}.
For any convex cone $K\subset \mathcal Y$, denote by $K^\circ:=\{z'\in \mathcal Y\mid \langle z',z\rangle\leq 0, \forall z\in K\}$ the polar of $K$.

The definition of isolated calmness below which is taken from \cite[Definition 2]{DSZ2017} is the most important concept in this paper.

\begin{definition}\label{def-isocalm}
	The set-valued mapping $G: \mathcal X\rightrightarrows \mathcal Y$ is said to be isolated calm at $\bar u$ for $\bar v$ if $\bar v\in G(\bar u)$ and there exist a positive constant $\kappa$ and neighborhoods $\mathcal U$ of $\bar u$ and $\mathcal V$ of $\bar v$ such that
	\begin{equation}\label{isocalm}
		G(u)\cap\mathcal V\subset\{\bar v\}+\kappa\|u-\bar u\|\mathbb B_\mathcal Y~~\forall u\in \mathcal U.
	\end{equation}
	Moreover, we say that $G$ is robustly isolated calm at $\bar u$ for $\bar v$ if (\ref{isocalm}) holds, and for each $u\in \mathcal U$, $G(u)\cap\mathcal V\neq\emptyset$.
\end{definition}

For any closed convex function $g: \mathcal X\rightarrow (-\infty,+\infty]$, we know from \cite[Proposition 2.58] {BS00} that  $g$ is directionally epidifferentiable. We use $g^\downarrow(x;\cdot)$ to denote the directional epiderivatives of $g$. Further, if $g^\downarrow(x;d)$ is finite for $x\in {\rm dom} \,g$ and $d\in \mathcal X$, we define the lower second-order directional epiderivative of $g$ for any $w\in \mathcal X$ by
\begin{equation*}
	g^{\downarrow\downarrow}_-(x;d,w):=\liminf\limits_{\tau\downarrow 0\atop w'\rightarrow w}\frac{g(x+\tau d+\frac{1}{2}\tau^2w')-g(x)-\tau g^\downarrow(x;d)}{\frac{1}{2}\tau^2}.
\end{equation*}
\section{Variational Analysis of the Spectral Norm Function}
In this section, we will provide a direct expression for the directional derivative of the proximal mapping of the spectral norm function. Then discuss the relationship between directional derivatives of the proximal mapping, the so-called ``sigma term" and the critical cones.

Given an arbitrary matrix $Q\in\Re^{m\times n}$, let $\sigma_1(Q)\geq\sigma_2(Q)\geq\cdots\geq\sigma_m(Q)$ be the singular values of $Q$. Denote $\sigma(Q):=(\sigma_1(Q),\sigma_2(Q),\cdots,\sigma_m(Q))^T$. For any integer $p>0$, let $\mathcal O^p$ be the set of all $p\times p$ orthogonal matrices. We assume that $Q\in\Re^{m\times n}$ admits the singular value decomposition (SVD) as follows:
\begin{equation}\label{svd}
	Q=U[ \begin{array}{cc}
		{\rm Diag}(\sigma(Q)) & 0 \\
	\end{array}
	]V^T=
	U[ \begin{array}{cc}
		{\rm Diag}(\sigma(Q)) & 0 \\
	\end{array}
	][ \begin{array}{cc}
		V_1 & V_2 \\
	\end{array} ]^T=U{\rm Diag}(\sigma(Q))V_1^T,
\end{equation}
	where $U\in \mathcal O^m$ and $[V_1\quad V_2]\in \mathcal O^n$ with $V_1\in \Re^{n\times m}$ and $V_2\in \Re^{n\times(n-m)}$. Define the following three index sets
\begin{equation}\label{indexabc}
	a:=\{1\leq i\leq m\mid\sigma_i(Q)>0\},\quad b:=\{1\leq i\leq m\mid\sigma_i(Q)=0\},\quad c:=\{m+1,\cdots,n\}.
\end{equation}

For any two matrices $P, W\in\mathcal S^n$, the well known Fan's inequality \cite{Fan1949} takes the form
\begin{equation}\label{fanineq}
	\langle P,W\rangle\leq \lambda(P)^T\lambda(W),
\end{equation}
where $\lambda(P)$ represents the eigenvalue vector of $P$ whose elements are arranged in nonincreasing order.

	For any $P,W\in \Re^{m\times n}$, by \cite[Theorem 31.5]{Roc70} we know that $W\in \partial\theta(P)$ (or, equivalently, $P\in \partial \theta^*(W)$) holds if and only if
\begin{equation*}
	{\rm Prox}_\theta(P+W)=P\iff{\rm Prox}_{\theta^*}(P+W)=W,
\end{equation*}
where ${\rm Prox}_\theta:\Re^{m\times n}\rightarrow\Re^{m\times n}$ denote the proximal mapping of $\theta$ (see, e.g., \cite[Definition 6.1]{beck2017}), namely,
\begin{equation*}\label{prox}
	{\rm Prox}_\theta(Z):=\arg{\min\limits_{Z'\in \Re^{m\times n}}}\left\{\theta(Z')+\frac{1}{2}\|Z'-Z\|^2\right\}, \quad Z\in\Re^{m\times n},
\end{equation*}
and ${\rm Prox}_{\theta^*}:\Re^{m\times n}\rightarrow\Re^{m\times n}$ denote the proximal mapping of $\theta^*$.
Denote $Q:=P+W$. Let $Q$ have the SVD as (\ref{svd}).
Then by \cite[Theorem 7.29 and Example 7.31]{beck2017} we have that
\begin{equation*}\label{moreaudes1}
	P=U[{\rm Diag}(\sigma(P))\quad 0]V^T,\quad W=U[{\rm Diag}(\sigma(W))\quad 0]V^T,
\end{equation*}
and
\begin{equation*}\label{moreaudes2}
	\sigma(P)=\kappa(\sigma(Q))\quad {\text{and}}\quad \sigma(W)=\sigma(Q)-\kappa(\sigma(Q)),
\end{equation*}
where $\kappa:\Re^m\rightarrow \Re^m$ is the proximal mapping of the $l_{\infty}$ norm (i.e., the maximum absolute value of the elements in a vector in $\Re^m$).

Define $\phi:\Re\rightarrow \Re$ as the following scalar function
\begin{equation*}\label{lowner}
	\phi(x):=\min\{x,\lambda^*\},\quad x\in\Re,
\end{equation*}
where $\lambda^*>0$ and satisfies $\sum\limits_{i=1}^m[\sigma_i(Q)-\lambda^*]_+=1$ (for any $x\in \Re$, $[x]_+$ means the nonnegative part of $x$).
Then, we can conclude that $\kappa(\sigma)=(\phi(\sigma_1),\phi(\sigma_2),\cdots,\phi(\sigma_m))$ and
\begin{equation*}\label{Ades}
	P={\rm Prox}_\theta(Q)=U[{\rm Diag}(\phi(\sigma_1),\phi(\sigma_2),\cdots,\phi(\sigma_m))\quad 0]V^T.
\end{equation*}
Obviously, ${\rm Prox}_\theta(\cdot)$ can be regarded as L\"{o}wner's operator related to $\phi$.

Let $\nu_1(Q),\nu_2(Q),\cdots,\nu_r(Q)$ be the nonzero different singular values of $Q$,
and there exists an integer $1\leq\tilde r\leq r$ such that
\begin{equation*}
	\nu_1(Q)>\nu_2(Q)>\cdots>\nu_{\tilde r}(Q)\geq \lambda^*>\nu_{\tilde r+1}(Q)>\cdots>\nu_r(Q)>0.
\end{equation*}
To proceed, we further define several index sets. We denote $a_{r+1}:=b$ for convenience.
Divide the set $a$ into subsets as follows:
\begin{equation*}\label{al}
	a_k:=\{i\in a\mid\sigma_i(Q)=\nu_k(Q)\},\quad k\in [r],
\end{equation*}
and define two index sets associated with $\lambda^*$:
\begin{equation}\label{alpha}
	\alpha:=\bigcup_{k=1}^{\tilde r}a_k,\quad \beta:=\bigcup_{k=\tilde r+1}^{r+1}a_k.
\end{equation}
Specifically, $\sigma_i(P)=\min\{\sigma_i(Q),\lambda^*\}$, $i=1,\cdots,m$, namely,
\begin{equation*}
	\sigma_i(P)=\left\{
	\begin{array}{ll}
		\lambda^*, & {\rm if}\; i\in \alpha, \\
		\sigma_i(Q), & {\rm if}\; i\in \beta.
	\end{array}\right.
\end{equation*}
In fact, $\lambda^*$ coincides with the largest singular value of $P$, namely, $\lambda^*=\|P\|_2$. According to Watson \cite{Watson1992}, the subdifferential of $\theta$ can be characterized as below:
\begin{equation}\label{subtheta}
	\partial \theta(P)=\{U_\alpha HV_{\alpha}^T,\forall H\in \mathcal S^{|\alpha|}_+,{\rm trace}(H)=1\}.
\end{equation}
Therefore, we know that $\sigma_i(W)=0$ when $i\in\beta$, and the set $\alpha$ can be divided into three subsets as follows:
\begin{equation}\label{alpha123}
	\alpha_1:=\{i\in \alpha\mid \sigma_i(W)=1\},\quad\alpha_2:=\{i\in \alpha\mid 0<\sigma_i(W)<1\},\quad \alpha_3:=\{i\in \alpha\mid \sigma_i(W)=0\}.
\end{equation}
Hence, there exist integers $0\leq {r_0}\leq 1$ and $\tilde r-1\leq {r_1}\leq \tilde r$ such that
\begin{equation*}
	\alpha_1=\bigcup_{k=1}^{ {r_0}}a_k,\quad \alpha_2=\bigcup_{k={r_0}+1}^{r_1}a_k,\quad \alpha_3=\bigcup_{k={r_1}+1}^{\tilde r}a_k.
\end{equation*}

Since the spectral norm function is globally Lipschitz continuous with modulus $1$ and convex, $\theta$ is directional differentiable at any point in $\Re^{m\times n}$. From (\ref{subtheta}), for any $D\in \Re^{m\times n}$, the directional derivative of $\theta$ at $P$ in the direction $D$ can be expressed as
\begin{equation}\label{dird}
	\theta'(P;D)=\sup\limits_{S\in\partial \theta(P)}\langle D,S\rangle=\|U_\alpha^TDV_\alpha\|_2.
\end{equation}
For $W\in \partial \theta(P)$, or, equivalently, $P\in\partial \theta^*(W)$, the critical cone of $\theta$ at $P$ for $W$ and the critical cone of $\theta^*$ at $W$ for $P$ are defined as
\begin{equation}\label{critconeA}
	\mathcal C_\theta(P,W):=\left\{D\in \Re^{m\times n}\mid \theta'(P;D)=\langle D,W\rangle\right\}
\end{equation}
and
\begin{equation}\label{critconeB}
	\mathcal C_{\theta^*}(W,P):=\left\{D\in \Re^{m\times n}\mid (\theta^*)'(W;D)=\langle D,P\rangle\right\}.
\end{equation}

Next we will give the expression for the directional derivative of ${\rm Prox}_\theta(\cdot)$. By \cite{DSST2018}, we know that ${\rm Prox}_\theta(\cdot)$ is directional differentiable and the directional derivative can be obtained via the directional derivative of $\phi$. Clearly the directional derivative of $\phi$ is
\begin{equation*}\label{phi'}
	\phi'(p;h)=\left\{
	\begin{array}{ll}
		0, & {\rm if}\; p>\lambda^*, \\[10pt]
		\min\{h,0\},&{\rm if}\; p=\lambda^*,\\[10pt]
		h, & {\rm if}\; p<\lambda^*.
	\end{array}\right.
\end{equation*}
	For any $Q\in \Re^{m\times n}$, denote $\Theta^2_{\alpha\alpha}:\Re^{m\times n}\rightarrow \Re^{|\alpha|\times|\alpha|}$, $\Theta^1_{\alpha\beta}:\Re^{m\times n}\rightarrow \Re^{|\alpha|\times|\beta|}$, $\Theta^2_{\alpha\beta}:\Re^{m\times n}\rightarrow \Re^{|\alpha|\times|\beta|}$ and $\Theta_{\alpha c}:\Re^{m\times n}\rightarrow \Re^{|\alpha|\times |c|}$ as
\begin{equation*}\label{Theta}
	\left\{
	\begin{array}{ll}
		(\Theta^2_{\alpha\alpha}(Q))_{ij}:=\displaystyle{\frac{2\lambda^*}{\sigma_i(Q)+\sigma_j(Q)}}, & i\in[|\alpha|], j\in[|\alpha|], \\[15pt]
		(\Theta^1_{\alpha\beta}(Q))_{ij}:=\displaystyle{\frac{\lambda^*-\sigma_{j+|\alpha|}(Q)}{\sigma_i(Q)-\sigma_{j+|\alpha|}(Q)}}, & i\in[|\alpha|], j\in[|\beta|], \\[15pt]
		(\Theta^2_{\alpha\beta}(Q))_{ij}:=\displaystyle{\frac{\lambda^*+\sigma_{j+|\alpha|}(Q)}{\sigma_i(Q)+\sigma_{j+|\alpha|}(Q)}}, & i\in[|\alpha|], j\in[|\beta|], \\[15pt]
		(\Theta_{\alpha c}(Q))_{ij}:=\displaystyle{\frac{\lambda^*}{\sigma_i(Q)}}, & i\in[|\alpha|], j\in[n-m].
	\end{array}\right.
\end{equation*}
Define two linear matrix operators $S:\Re^{p\times p}\rightarrow \mathcal S^p$ and $T:\Re^{p\times p}\rightarrow \Re^{p\times p}$ by
\begin{equation}\label{S&T}
	S(H):=\frac{1}{2}(H+H^T)\quad{\text {and}}\quad T(H):=\frac{1}{2}(H-H^T),\forall H\in\Re^{p\times p}.
\end{equation}
For all $Q\in \Re^{m\times n}$ and $D\in\Re^{m\times n}$, let $D=[D_1\quad D_2]$ with $D_1\in \Re^{m\times m}$ and $D_2\in\Re^{m\times (n-m)}$.
We define four matrix mappings $\Xi_1:\Re^{m\times n}\times \Re^{m\times n}\rightarrow \Re^{|\alpha|\times |\alpha|}$, $\Xi_2:\Re^{m\times n}\times \Re^{m\times n}\rightarrow \Re^{|\alpha|\times |\beta|}$, $\Xi_3:\Re^{m\times n}\times \Re^{m\times n}\rightarrow \Re^{|\beta|\times |\alpha|}$ and $\Xi_4:\Re^{m\times n}\times \Re^{m\times n}\rightarrow \Re^{|\alpha|\times |c|}$ by
\begin{equation}\label{Xi}
	\left\{
	\begin{array}{l}
		\Xi_1(Q,D):=\Theta^2_{\alpha\alpha}(Q)\circ T(D_1)_{\alpha\alpha},\\[15pt]
		\Xi_2(Q,D):=\Theta^1_{\alpha\beta}(Q)\circ S(D_1)_{\alpha\beta}+ \Theta^2_{\alpha\beta}(Q)\circ T(D_1)_{\alpha\beta},\\[15pt]
		\Xi_3(Q,D):=(\Theta^1_{\alpha\beta}(Q))^T\circ S(D_1)_{\beta\alpha}+ (\Theta^2_{\alpha\beta}(Q))^T\circ T(D_1)_{\beta\alpha}, \\[15pt]
		\Xi_4(Q,D):=\Theta_{\alpha c}(Q)\circ D_{\alpha c}. \\
	\end{array}
	\right.
\end{equation}
Therefore, according to \cite[Theorem 3]{DSST2018}, the directional derivative ${\rm Prox}_\theta'(Q;D)$ can be written as
\begin{equation}\label{prox'}
	{\rm Prox}_\theta'(Q;D)=U
	\left(
	\begin{array}{c;{2pt/2pt}c;{2pt/2pt}c}
		\Xi_1(Q,\widetilde D)+\Xi_{\alpha\alpha}(\widetilde D)
		& \Xi_2(Q,\widetilde D)&  \Xi_4(Q,\widetilde D) \\[15pt]
		\hdashline[2pt/2pt]
		&&\\
		\Xi_3(Q,\widetilde D) & {\widetilde D}_{\beta\beta} &{\widetilde D}_{\beta c}\\
	\end{array}
	\right)
	V^T,
\end{equation}
where $\widetilde D:=U^TDV$ and
\begin{equation}\label{Xialpha}
	\Xi_{\alpha\alpha}(\widetilde D):=\left(
	\begin{array}{ccc}
		0_{\alpha_1\times\alpha_1} &  &  \\
		& 0_{\alpha_2\times\alpha_2} &  \\
		&  & \Pi_{\mathcal S_-^{|\alpha_3|}}(S({\widetilde D}_{\alpha_3\alpha_3})) \\
	\end{array}
	\right).
\end{equation}

Theorem 3.1 in \cite{ZZX2013} shows that the singular value function is second order directionally differentiable in $\Re^{m\times n}$. This means that the spectral norm function is second order directional differentiable, and $\theta''(X;\cdot)=\theta^{\downdownarrows}(X;\cdot)$. To analyze the SOSC of the problem (\ref {p}), which helps to characterize the ICKKTM, we need to compute the ``sigma term" of the problem (\ref{p}), namely,
the conjugate function of the second order directional derivative of $\theta$.
For any $k=1,\cdots,r$, define $\Omega_{a_k}:\Re^{m\times n}\times \Re^{m\times n}\rightarrow \Re^{a_k\times a_k}$ by
\begin{equation*}
	\begin{split}
		\Omega_{a_k}(P,D):= & \left(S(\widetilde D_1)\right)^T_{a_k}\Big(\Sigma(P)-\nu_k(P)I_m\Big)^\dagger\left(S(\widetilde D_1)\right)_{a_k}-(2\nu_k(P))^{-1}\widetilde D_{a_kc}\widetilde D_{a_kc}^T \\
		& +\left(T(\widetilde D_1)\right)^T_{a_k}\Big(-\Sigma(P)-\nu_k(P)I_m\Big)^\dagger\left(T(\widetilde D_1)\right)_{a_k},
	\end{split}
\end{equation*}
where $Z^\dagger$ is the Moore-Penrose pseudo-inverse of $Z$, $I_m$ is the $m\times m$ identity matrix and $\widetilde D=[\widetilde D_1\quad \widetilde D_2]=[U^TDV_1\quad U^TDV_2]$.
By \cite{Ding2017}, we know that $\forall D\in \Re^{m\times n}$, the conjugate of $\theta''(P;D,\cdot)$ is
\begin{equation}\label{sigmaterm}
	\psi^*_{(P,D)}(W):=(\theta''(P;D,\cdot))^*(W)=\langle (\Sigma(W))_{\alpha\alpha},2\Omega_\alpha(P,D) \rangle
	=\sum_{k=1}^{r_1}\sigma_k(W){\rm trace}(2\Omega_{a_k}(P,D)),
\end{equation}
where $\sigma_k(W)=(\sigma_i(W))_{i\in a_k}$.
Clearly $\nu_k(P)=\lambda^*$ when $k=1,\cdots, r$. Continuing to calculate formula (\ref{sigmaterm}), we can get the explicit expression for $\psi^*_{(P,D)}(W)$ as follows
\begin{equation*}
	\begin{split}
		\psi^*_{(P,D)}(W)= &\sum_{\substack{1\leq l\leq r_1\\\widetilde r+1\leq t\leq r+1}}
		\frac{2\sigma_l(W)}{\nu_t(P)-\lambda^*}\|(S(\widetilde D_1))_{a_la_t}\|^2-\sum_{1\leq l\leq r_1}\frac{\sigma_l(W)}{\lambda^*}\|\widetilde D_{a_lc}\|^2 \\
		& +\sum_{\substack{1\leq l\leq r_1\\1\leq t\leq r+1}}
		\frac{2\sigma_l(W)}{-\nu_t(P)-\lambda^*}\|(T(\widetilde D_1))_{a_la_t}\|^2.
	\end{split}
\end{equation*}

We are now prepared to give the main conclusions of this section, that is, some properties of the critical cones and the conjugate function of second order directional derivative $\psi^*$.

Propositions 10 and 12 in \cite{Ding2017} characterized the critical cones of $\theta$ and $\theta^*$  defined in (\ref{critconeA}) and (\ref{critconeB}), respectively. We summarize the results as follows.
\begin{proposition}\label{crit12}
	Let $W\in\partial \theta(P)$ and $Q=P+W$ admit the SVD as in (\ref{svd}). Let the index sets $\alpha$, $\beta$, $\alpha_1$, $\alpha_2$, $\alpha_3$ and $c$ be defined as (\ref{indexabc}), (\ref{alpha}) and (\ref{alpha123}). Given any $D\in \Re^{m\times n}$, denote $\widetilde D=U^TDV$. Then the following results hold:
	\begin{enumerate}[align=left, label=\rm ({\roman*})]
		\item $D\in \mathcal C_\theta(P,W)$ if and only if there exists some $\tau\in\Re$ such that
		\begin{equation*}
			\lambda_{|\alpha_1|}(S(\widetilde D_{\alpha_1\alpha_1}))\geq \tau\geq \lambda_1(S(\widetilde D_{\alpha_3\alpha_3}))
		\end{equation*}
		and
		\begin{equation*}
			S(\widetilde D_{\alpha\alpha})=\left(
			\begin{array}{ccc}
				S(\widetilde D_{\alpha_1\alpha_1}) &  &  \\
				& \tau I_{|\alpha_2|} &  \\
				&  & S(\widetilde D_{\alpha_3\alpha_3}) \\
			\end{array}
			\right),
		\end{equation*}
		where $\lambda(Z)$ denote the real eigenvalue vector of symmetric matrix $Z$ being arranged in nonincreasing order.
		\item   $D\in \mathcal C_{\theta^*}(W,P)$ if and only if
		\begin{equation*}
			{\rm trace}(\widetilde D_{\alpha\alpha})=0,\quad S(\widetilde D_{\alpha_1\alpha_1})\in\mathcal S_-^{|\alpha_1|},\quad \widetilde D_{\alpha_3\alpha_3}\in \mathcal S_+^{|\alpha_3|}
		\end{equation*}
		and
		\begin{equation*}
			\widetilde D=\left(
			\begin{array}{ccccc}
				\widetilde D_{\alpha_1\alpha_1} & \widetilde D_{\alpha_1\alpha_2} & \widetilde D_{\alpha_1\alpha_3} & \widetilde D_{\alpha_1\beta} & \widetilde D_{\alpha_1c} \\[4pt]
				\widetilde D_{\alpha_2\alpha_1} & \widetilde D_{\alpha_2\alpha_2} & \widetilde D_{\alpha_2\alpha_3} & \widetilde D_{\alpha_2\beta} & \widetilde D_{\alpha_2c} \\[4pt]
				\widetilde D_{\alpha_3\alpha_1} & \widetilde D_{\alpha_3\alpha_2} & \widetilde D_{\alpha_3\alpha_3} & 0 & 0 \\[4pt]
				\widetilde D_{\beta\alpha_1} & \widetilde D_{\beta\alpha_2} & 0 & 0 & 0 \\[2pt]
			\end{array}
			\right).
		\end{equation*}
	\end{enumerate}
\end{proposition}
\begin{remark}
	For the spectral norm function $\theta$, the $\tau$ in part (i) of proposition \ref{crit12} is actually the largest eigenvalue of $S(\widetilde D_{\alpha\alpha})$, i.e., $\tau=\lambda_1(S(\widetilde D_{\alpha\alpha}))$. This means that $\tau$ is relevant to $D$. It is easy to observe that the index sets $\alpha_1$ and $\alpha_2$ cannot exist simultaneously. When $\alpha_1\neq \emptyset$, there must be $|\alpha_1|=1$ and $\alpha_2=\emptyset$, then $D\in \mathcal C_\theta(P,W)$ if and only if
	\begin{equation*}
		S(\widetilde D_{\alpha\alpha})=\left(
		\begin{array}{cc}
			\lambda_1(S(\widetilde D_{\alpha\alpha})) &  \\
			& S(\widetilde D_{\alpha_3\alpha_3}) \\
		\end{array}
		\right).
	\end{equation*}
	When $\alpha_1=\emptyset$, $\alpha_2\neq\emptyset$, $D\in \mathcal C_\theta(P,W)$ if and only if
	\begin{equation*}
		S(\widetilde D_{\alpha\alpha})=\left(
		\begin{array}{cc}
			\lambda_1(S(\widetilde D_{\alpha\alpha}))I_{|\alpha_2|} &  \\
			& S(\widetilde D_{\alpha_3\alpha_3}) \\
		\end{array}
		\right).
	\end{equation*}
	In brief, we can conclude that $D\in \mathcal C_\theta(P,W)$ if and only if
	\begin{equation}\label{critc0}
		S(\widetilde D_{\alpha\alpha})=\left(
		\begin{array}{cc}
			\lambda_1(S(\widetilde D_{\alpha\alpha}))I_{|\alpha_1|+|\alpha_2|} &  \\
			& S(\widetilde D_{\alpha_3\alpha_3}) \\
		\end{array}
		\right).
	\end{equation}
\end{remark}
\begin{proposition}\label{3.2}
	Let $W\in\partial \theta(P)$ and $Q=P+W$ admit the SVD as in (\ref{svd}). Let the index sets $\alpha$, $\beta$, $\alpha_1$, $\alpha_2$, $\alpha_3$ and $c$ be defined as (\ref{indexabc}), (\ref{alpha}) and (\ref{alpha123}). Given any $H\in \Re^{m\times n}$, denote $\widetilde H=U^THV$. Then the following results hold:
	\begin{enumerate}[align=left, label=\rm ({\roman*})]
		\item
		If $H\in \mathcal C_{\theta^*}(W,P)$ and $\varphi^*_{(W,H)}(P)=0$, then $H\in (\mathcal C_\theta(P,W))^\circ$.
		\item
	    If $H\in \mathcal C_{\theta}(P,W)$ and $\psi^*_{(P,H)}(W)=0$ if and only if $H\in (\mathcal C_{\theta^*}(P,W))^\circ$.
	\end{enumerate}
\end{proposition}
	\begin{proof} We first prove part (i). By \cite[Proposition 16]{Ding2017}, $\forall H\in\Re^{m\times n}$, $\psi^*_{(P,H)}(W)=0$ if and only if  $\varphi^*_{(W,H)}(P)=0$, which also is equivalent to
	\begin{equation*}
		\left(
		\begin{array}{cc}
			\widetilde H_{\alpha_1\alpha_1} & \widetilde H_{\alpha_1\alpha_2} \\
			\widetilde H_{\alpha_2\alpha_1} & \widetilde H_{\alpha_2\alpha_2} \\
		\end{array}
		\right)\in \mathcal S^{|\alpha_1|+|\alpha_2|},\quad
		\widetilde H_{\alpha_1\alpha_3}=(\widetilde H_{\alpha_3\alpha_1})^T,\quad  \widetilde H_{\alpha_2\alpha_3}=(\widetilde H_{\alpha_3\alpha_2})^T,
	\end{equation*}
	\begin{equation*}
		\widetilde H_{\alpha_1\beta}=(\widetilde H_{\beta\alpha_1})^T=0,\quad \widetilde H_{\alpha_2\beta}=(\widetilde H_{\beta\alpha_2})^T=0,\quad
		\widetilde H_{\alpha_1c}=0, \quad \widetilde H_{\alpha_2c}=0.
	\end{equation*}
	Then, it follows from part (ii) of Proposition \ref{crit12} that $H\in\mathcal C_{\theta^*}(W,P)$ and $\varphi^*_{(W,H)}(P)=0$ are equivalent to
	\begin{equation}\label{polarcrit}
		\widetilde H_{\alpha\alpha}\in \mathcal S^{|\alpha|},\quad {\rm trace}(\widetilde H_{\alpha\alpha})=0,\quad H_{\alpha_1\alpha_1}\in \mathcal S_-^{|\alpha_1|},\quad H_{\alpha_3\alpha_3}\in \mathcal S_+^{|\alpha_3|}\quad {\rm and}\quad
		\widetilde H=\left(
		\begin{array}{ccc}
			\widetilde H_{\alpha\alpha} & 0 & 0 \\
			0 & 0 & 0 \\
		\end{array}
		\right).
	\end{equation}
	By part (i) of Proposition \ref{crit12}, for any $D\in \mathcal C_\theta(P,W)$, there exists $\tau'\in\Re$ such that
	\begin{equation*}
		\lambda_{|\alpha_1|}(S(\widetilde D_{\alpha_1\alpha_1}))\geq \tau'\geq \lambda_1(S(\widetilde D_{\alpha_3\alpha_3}))
	\end{equation*}
	and
	\begin{equation*}
		S(\widetilde D_{\alpha\alpha})=\left[
		\begin{array}{ccc}
			S(\widetilde D_{\alpha_1\alpha_1}) &  &  \\
			& \tau' I_{|\alpha_2|} &  \\
			&  & S(\widetilde D_{\alpha_3\alpha_3}) \\
		\end{array}
		\right].
	\end{equation*}
	Hence, $\forall D\in \mathcal C_\theta(P,W)$ and $H$ satisfies ( \ref{polarcrit}), based on the Fan's inequality (\ref{fanineq}), we have that
	\begin{equation*}
		\begin{split}
			\langle H,D \rangle &=\langle \widetilde H,\widetilde D \rangle=\langle \widetilde H_{\alpha\alpha},\widetilde D_{\alpha\alpha} \rangle
			=\langle S(\widetilde H_{\alpha\alpha}), \widetilde D_{\alpha\alpha} \rangle=\langle \widetilde H_{\alpha\alpha}, S(\widetilde D_{\alpha\alpha}) \rangle\\[4pt]
			& =\langle \widetilde H_{\alpha_1\alpha_1}, S(\widetilde D_{\alpha_1\alpha_1}) \rangle+\tau'{\rm trace}(\widetilde H_{\alpha_2\alpha_2})+
			\langle \widetilde H_{\alpha_3\alpha_3}, S(\widetilde D_{\alpha_3\alpha_3})\\[4pt]
			&\leq (\lambda(\widetilde H_{\alpha_1\alpha_1}))^T\lambda(S(\widetilde D_{\alpha_1\alpha_1}))+\tau'{\rm trace}(\widetilde H_{\alpha_2\alpha_2})+(\lambda(\widetilde H_{\alpha_3\alpha_3}))^T\lambda(S(\widetilde D_{\alpha_3\alpha_3}))\\[4pt]
			&=(-\lambda(\widetilde H_{\alpha_1\alpha_1}))^T(-\lambda(S(\widetilde D_{\alpha_1\alpha_1}))+\tau'{\rm trace}(\widetilde H_{\alpha_2\alpha_2})+(\lambda(\widetilde H_{\alpha_3\alpha_3}))^T\lambda(S(\widetilde D_{\alpha_3\alpha_3}))\\[4pt]
			&\leq \tau'{\rm trace}(\widetilde H_{\alpha_1\alpha_1})+\tau'{\rm trace}(\widetilde H_{\alpha_2\alpha_2})+\tau'{\rm trace}(\widetilde H_{\alpha_3\alpha_3})\\[4pt]
			&=\tau' {\rm trace}(\widetilde H_{\alpha\alpha})\\[4pt]
			&=0,
		\end{split}
	\end{equation*}
	which means that $H\in (\mathcal C_\theta(P,W))^\circ$.
Conversely, when $\tau'$ in $\mathcal C_\theta(P,W)$ equals to $0$, we can not get ${\rm trace}(\widetilde H_{\alpha\alpha})=0$, i.e., the reverse conclusion is not true. Then the proof of part (i) is completed.
	
	To prove (ii), it is easy to see that $\psi^*_{(P,H)}(W)=0$ and $H\in \mathcal C_{\theta}(P,W)$ hold if and only if there exists $\tau''\in\Re$ such that
	\begin{equation*}
		\lambda_{|\alpha_1|}(S(\widetilde H_{\alpha_1\alpha_1}))\geq \tau''\geq \lambda_1(S(\widetilde H_{\alpha_3\alpha_3}))
	\end{equation*}
	and
	\begin{equation*}
		\widetilde H=\left(
		\begin{array}{ccccc}
			S(\widetilde H_{\alpha_1\alpha_1}) & 0 & 0 & 0& 0 \\[4pt]
			0 & \tau'' I_{|\alpha_2|} & 0  & 0& 0\\[4pt]
			0& 0  & \widetilde H_{\alpha_3\alpha_3} & \widetilde H_{\alpha_3\beta} &\widetilde H_{\alpha_3c}\\[4pt]
			0& 0  & \widetilde H_{\beta\alpha_3} & \widetilde H_{\beta\beta} &\widetilde H_{\beta_3c}
		\end{array}
		\right),
	\end{equation*}
	which can be proved to be equivalent to $H\in (\mathcal C_{\theta^*}(W,P))^\circ$ by using the same method as in part (i).
\end{proof}
The following propositions provide some properties of directional derivatives of the proximal mapping of $\theta$.

\begin{proposition}\label{3.4}
	Let $W\in\partial \theta(P)$ and $Q=P+W$ admit the SVD as in (\ref{svd}). Let the index sets $\alpha$, $\beta$, $\alpha_1$, $\alpha_2$, $\alpha_3$ and $c$ be defined as (\ref{indexabc}), (\ref{alpha}) and (\ref{alpha123}). Given any $D\in \Re^{m\times n}$, denote $\widetilde D=U^TDV$. The function $\psi^*_{(P,D)}(W)$ is defined in (\ref{sigmaterm}). For any $D\in \Re^{m\times n}$,  if $H={\rm Prox}'_\theta(Q;H+D)$ holds, then $H\in \mathcal C_\theta(P,W)$ and $\langle H,D\rangle=-\psi^*_{(P,D)}(W)$.
\end{proposition}
\begin{proof}
	Suppose $H={\rm Prox}'_\theta(Q;H+D)$ holds, together with the expression for ${\rm Prox}'_\theta(Q;\cdot)$ in (\ref{prox'}), we have that
	\begin{equation}\label{prox'hd}
		\begin{split}
			\widetilde H&=U^T{\rm Prox}'_\theta(Q;H+D)V\\[4pt]
			&=\left(
			\begin{array}{c;{2pt/2pt}c;{2pt/2pt}c}
				\Xi_1(Q,\widetilde H+\widetilde D)+\Xi_{\alpha\alpha}(\widetilde H+\widetilde D)
				& \Xi_2(Q,\widetilde H+\widetilde D)&  \Xi_4(Q,\widetilde H+\widetilde D)\\[10pt]
				\hdashline[2pt/2pt]
&\\
				\Xi_3(Q,\widetilde H+\widetilde D) & {\widetilde H}_{\beta\beta}+ {\widetilde D}_{\beta\beta}&{\widetilde H}_{\beta c}+{\widetilde D}_{\beta c}\\
			\end{array}
			\right),
		\end{split}
\end{equation}
	where $\Xi_{\alpha\alpha}$ is defined in (\ref{Xialpha}). We first verify $H\in \mathcal C_\theta(P,W)$.
The equality
\begin{equation*}
	\begin{split}
		\widetilde H_{\alpha\alpha}&=\Xi_1(Q,\widetilde H+\widetilde D)+\Xi_{\alpha\alpha}(\widetilde H+\widetilde D)\\
		&=\Theta^2_{\alpha\alpha}(Q)\circ \left(T(\widetilde H_1)_{\alpha\alpha}+(T(\widetilde D_1)_{\alpha\alpha}\right)+
		\left(
		\begin{array}{ccc}
			0_{\alpha_1\times\alpha_1} &  &  \\
			& 0_{\alpha_2\times\alpha_2} &  \\
			&  & \Pi_{\mathcal S_-^{|\alpha_3|}}(S({\widetilde H}_{\alpha_3\alpha_3})+S({\widetilde D}_{\alpha_3\alpha_3})) \\
		\end{array}
		\right)
	\end{split}
\end{equation*}
implies that
\begin{equation*}
	S(\widetilde H)=\left(
	\begin{array}{ccc}
		0_{\alpha_1\times\alpha_1} &  &  \\
		& 0_{\alpha_2\times\alpha_2} &  \\
		&  & \Pi_{\mathcal S_-^{|\alpha_3|}}(S({\widetilde H}_{\alpha_3\alpha_3})+S({\widetilde D}_{\alpha_3\alpha_3})) \\
	\end{array}
	\right).
\end{equation*}
From part (i) of proposition \ref{crit12}, we obtain that $H\in \mathcal C_\theta(P,W)$.

	For convenience, we denote $\alpha':=\alpha_1\cup\alpha_2$. By directly calculating (\ref{prox'hd}), we deduce that $H={\rm Prox}'_\theta(Q;H+D)$ if and only if
		\begin{equation}\label{Dij}
		\left\{
		\begin{array}{ll}
			S(\widetilde H_{\alpha'\alpha'})=0,\quad S(\widetilde H_1)_{\alpha'\alpha_3}=0,\quad S(\widetilde H_1)_{\alpha_3\alpha'}=0,\\[10pt]
			\mathcal S_-^{|\alpha_3|}\ni \mathcal S(\widetilde H_{\alpha_3\alpha_3})\perp\mathcal S(\widetilde D_{\alpha_3\alpha_3})= \widetilde D_{\alpha_3\alpha_3}\in \mathcal S_+^{|\alpha_3|}, & \\[10pt]
			T(\widetilde D_{\alpha'\alpha'})_{ij}=\left(\displaystyle\frac{\sigma_i(Q)+\sigma_j(Q)}{2\lambda^*}-1\right)(\widetilde H_{\alpha'\alpha'})_{ij}, & i\in[|\alpha'|], j\in[|\alpha'|], \\[10pt]
			(T(\widetilde D_1)_{\alpha'\alpha_3})_{ij}=\displaystyle\frac{\sigma_i(Q)-\lambda^*}{2\lambda^*}(\widetilde H_{\alpha'\alpha_3})_{ij}, & i\in[|\alpha'|], j\in[|\alpha_3|], \\[10pt]
			(T(\widetilde D_1)_{\alpha_3\alpha'})_{ji}=\displaystyle\frac{\sigma_i(Q)-\lambda^*}{2\lambda^*}(\widetilde H_{\alpha_3\alpha'})_{ji}, & i\in[|\alpha'|], j\in[|\alpha_3|], \\[10pt]
			(\widetilde D_{\alpha\beta})_{ij}=
			\displaystyle\frac{\lambda^*(\sigma_i(Q)-\lambda^*)}{{\lambda^*}^2-(\sigma_{j+|\alpha|}(Q))^2}(\widetilde H_{\alpha\beta})_{ij}+
			\displaystyle\frac{\sigma_{j+|\alpha|}(Q)(\sigma_i(Q)-\lambda^*)}{{\lambda^*}^2-(\sigma_{j+|\alpha|}(Q))^2}(\widetilde H_{\beta\alpha})_{ji}, &i\in[|\alpha|], j\in[|\beta|],\\[10pt]
			(\widetilde D_{\beta\alpha})_{ji}=
			\displaystyle\frac{\lambda^*(\sigma_i(Q)-\lambda^*)}{{\lambda^*}^2-(\sigma_{j+|\alpha|}(Q))^2}(\widetilde H_{\beta\alpha})_{ji}+
			\displaystyle\frac{\sigma_{j+|\alpha|}(Q)(\sigma_i(Q)-\lambda^*)}{{\lambda^*}^2-(\sigma_{j+|\alpha|}(Q))^2}(\widetilde H_{\alpha\beta})_{ij}, &i\in[|\alpha|], j\in[|\beta|],\\[10pt]
			(\widetilde D_{\alpha c})_{ij}=\left(\displaystyle\frac{\sigma_i(Q)}{\lambda^*}-1\right)(\widetilde H_{\alpha c})_{ij}, & i\in[|\alpha|], j\in[ n-m],\\[10pt]
			\widetilde D_{\beta\beta}=0,\\[10pt]
			\widetilde D_{\beta c}=0.
		\end{array}\right.
	\end{equation}
		Noting that $\nu_k(Q)-\lambda^*=\sigma_k(W)$ for $1\leq k\leq r_1$, $\nu_t(P)=\lambda^*$ for $1\leq t\leq \widetilde r$ and $\nu_t(P)=\nu_t(Q)$ for $\widetilde r+1\leq t\leq r+1$. Together with (\ref{Dij}) and the fact that for any $D\in\Re^{p\times p}$,
	\begin{equation*}
		\begin{array}{l}
			\langle D, D^T\rangle=\|S(D)\|^2-\|T(D)\|^2, \\[10pt]
			\langle D, D\rangle=\|D\|^2=\|S(D)\|^2+\|T(D)\|^2,
		\end{array}
	\end{equation*}
	we obtain that
	\begin{equation*}
		\begin{split}
			\langle H,D \rangle &=\langle \widetilde H,\widetilde D \rangle=\langle \widetilde H_{\alpha\alpha},\widetilde D_{\alpha\alpha}\rangle+
			\langle \widetilde H_{\alpha\beta},\widetilde D_{\alpha\beta}\rangle+\langle \widetilde H_{\beta\alpha},\widetilde D_{\beta\alpha}\rangle+
			\langle \widetilde H_{\alpha c},\widetilde D_{\alpha c} \rangle\\[10pt]
			&=\langle \widetilde H_{\alpha\alpha},T(\widetilde D_{\alpha\alpha})\rangle+
			\langle \widetilde H_{\alpha\beta},\widetilde D_{\alpha\beta}\rangle+\langle \widetilde H_{\beta\alpha},\widetilde D_{\beta\alpha}\rangle+
			\langle \widetilde H_{\alpha c},\widetilde D_{\alpha c} \rangle\\[10pt]&=\sum_{\substack{1\leq l\leq r_1\\ 1\leq t\leq \widetilde r}}
			\frac{\nu_l(Q)-\lambda^*}{\lambda^*}\|(T(\widetilde H_1))_{a_la_t}\|^2+\sum_{1\leq l\leq r_1}\frac{\nu_l(Q)-\lambda^*}{\lambda^*}\|\widetilde H_{a_lc}\|^2 \\[4pt]
			&\quad\quad+\sum_{\substack{1\leq l\leq r_1\\ \widetilde r+1\leq t\leq r+1}}\left(
			-\frac{2(\nu_l(Q)-\lambda^*)}{\nu_t(Q)-\lambda^*}\|S(\widetilde H_1)_{a_la_t}\|^2+\frac{2(\nu_l(Q)-\lambda^*)}{\nu_t(Q)+\lambda^*}\|T(\widetilde H_1)_{a_la_t}\|^2
			\right)\\[4pt]
			&=-\sum_{\substack{1\leq l\leq r_1\\ \widetilde r+1\leq t\leq r+1}}
			\frac{2\sigma_l(W)}{\nu_t(P)-\lambda^*}\|S(\widetilde H_1)_{a_la_t}\|^2+\sum_{1\leq l\leq r_1}\frac{\sigma_l(W)}{\lambda^*}\|\widetilde H_{a_lc}\|^2\\[4pt]
			&\quad\quad+\sum_{\substack{1\leq l\leq r_1\\ 1\leq t\leq r+1}}
			\frac{2\sigma_l(W)}{\nu_t(P)+\lambda^*}\|(T(\widetilde H_1)_{a_la_t}\|^2,\\[4pt]
		\end{split}
	\end{equation*}
	so we have
	\begin{equation*}
		\langle H,D \rangle =\langle \widetilde H,\widetilde D \rangle=-\psi^*_{(P,H)}(W).
	\end{equation*}
	This completes the proof.
\end{proof}
		\begin{proposition}\label{3.5}
		Let $W\in\partial \theta(P)$ and $Q=P+W$ admit the SVD as in (\ref{svd}). Let the index sets $\alpha$, $\beta$, $\alpha_1$, $\alpha_2$, $\alpha_3$ and $c$ be defined as (\ref{indexabc}), (\ref{alpha}) and (\ref{alpha123}). For any $D\in \Re^{m\times n}$, then
		\begin{equation*}
			{\rm Prox}'_\theta(Q;D)=0\Longleftrightarrow D\in(\mathcal C^0_\theta(P,W))^\circ,
		\end{equation*}
where $\mathcal C^0_\theta(P,W)$ represents the critical cone $\mathcal C_\theta(P,W)$ in the case of $\tau=0$, i.e.,
		\begin{equation*}
			\mathcal C^0_\theta(P,W)=\left\{H\in \Re^{m\times n}\mid S(\widetilde H_{\alpha\alpha})=\left(
			\begin{array}{cc}
				0_{|\alpha_1|+|\alpha_2|} &  \\
				& S(\widetilde H_{\alpha_3\alpha_3}) \\
			\end{array}
			\right),\quad S(\widetilde H_{\alpha_3\alpha_3})\in \mathcal S_-^{|\alpha_3|}\right\}.
		\end{equation*}
	\end{proposition}
	\begin{proof}
		It is easy to verify either ${\rm Prox}'_\theta(Q;D)=0$ or $D\in(\mathcal C^0_\theta(P,W))^\circ$
		is equivalent to
		\begin{equation*}
			\widetilde D=\left(
			\begin{array}{ccc}
				\widetilde D_{\alpha\alpha} & 0 & 0 \\
				0 & 0 & 0 \\
			\end{array}
			\right),\quad \widetilde D_{\alpha\alpha}\in \mathcal S^{|\alpha|},\quad D_{\alpha_3\alpha_3}\in \mathcal S_+^{|\alpha_3|}.
		\end{equation*}
	\end{proof}
	\section{The Robust ICKKTM}
In this section, owing to the special linear structure of the problem (\ref{p}), we first show the primal (dual) SRCQ is equivalent to the dual (primal) SOSC. In fact, when the spectral norm in problem (\ref{p}) is substituted by the nuclear norm, the same results are already given in \cite{CS2018}. Therefore, we only briefly state here, and the main tools we used are propositions \ref{crit12} and \ref{3.2}.
	The Lagrangian dual of problem (\ref{p}) is as follows:
\begin{equation}\label{d}
	\begin{array}{cl}
		\max\limits_{y,w,S} & -\langle b,y\rangle-\delta_{\mathcal P^\circ}(y)-h^*(w)\\
		{\rm s.t.} & {\mathcal B}^*y+{\mathcal Q}^*w+S+C=0,\quad \|S\|_*\leq 1,
	\end{array}
\end{equation}
where $(y,w,S)\in \Re^l\times\Re^d\times\Re^{m\times n}$, $\mathcal B^*$ and $\mathcal Q^*$ represent the adjoint of $\mathcal B$ and $\mathcal Q$, respectively.
The first order optimality conditions, i.e., the KKT conditions for the primal problem (\ref{p}) and the dual problem (\ref{d}) are given by
\begin{equation}\label{kktp}
	\left\{
	\begin{array}{l}
		\mathcal Q^*\nabla h(\mathcal Q X)+C+S+\mathcal B^*y=0, \\[5pt]
		\mathcal P^\circ \ni y\perp \mathcal BX-b\in \mathcal P,\\[5pt]
		S\in \partial \theta(X)
	\end{array}
	\quad(X,y,S)\in \Re^{m\times n}\times \Re^l\times \Re^{m\times n}
	\right.
\end{equation}
and
\begin{equation}\label{kktd}
	\left\{
	\begin{array}{l}
		\mathcal B^*y+\mathcal Q^*w+S+C=0, \\[5pt]
		\mathcal P^\circ \ni y\perp \mathcal BX-b\in \mathcal P,\\[5pt]
		\mathcal QX\in \partial h^*(w),\quad X\in \partial \theta^*(S),
	\end{array}
	\quad(y,w,S,X)\in \Re^l\times \Re^d\times \Re^{m\times n}\times \Re^{m\times n},
	\right.
\end{equation}
respectively.
We denote by $\mathcal M_P(X)$ the set of Lagrangian multipliers respect to $X$ for problem (\ref{p}), namely,
$\mathcal M_P(X):=\{(y, S)\in \Re^l\times \Re^{m\times n}\mid (X,y,S)\,\,{\rm {satisfies}} \,\,(\ref{kktp})\}$.
We write $\mathcal M_D(y,w,S)$ the set of lagrangian multipliers respect to $(y,w,S)$ for problem (\ref{d}), i.e., $\mathcal M_D(y,w,S):=\{X\in\Re^{m\times n}\mid(y,w,S,X)\,\,{\rm {satisfies}}\,\, (\ref{kktd})$\}.

Let $(y,z)\in\Re^l\times\Re^l$ satisfies $\mathcal P^\circ \ni y\perp z\in \mathcal P$. Define the critical cone of $\mathcal P$ at $z$ for $y$ as $\mathcal C_{\mathcal P}(z,y):=\mathcal T_{\mathcal P}(z)\cap y^\perp$ and the critical cone of $\mathcal P^\circ$ at $y$ for $z$ as $\mathcal C_{\mathcal P^\circ}(y,z):=\mathcal T_{\mathcal P^\circ}(y)\cap z^\perp$, respectively. Obviously,
\begin{equation}\label{cripor}
	(\mathcal C_{\mathcal P}(z,y))^\circ =\mathcal C_{\mathcal P^\circ}(y,z).
\end{equation}

	Now we are ready to demonstrate the crucial result of this paper, namely, the relationship between the primal SRCQ and the dual SOSC. Actually, this result was given in \cite[Theorem 5.2]{CS2018}, but they omitted the proof details. For the sake of completeness, we supplement the proof below.

\begin{theorem}\label{p-srcq}
	Let $(\bar y,\bar w,\overline S)\in \Re^l\times\Re^d\times\Re^{m\times n}$ be an optimal solution of problem (\ref{d}) with $\mathcal M_D(\bar y,\bar w,\overline S)\neq \emptyset$. Let $\overline X\in \mathcal M_D(\bar y,\bar w,\overline S)$. Then the following results are equivalent:
	\begin{itemize}
		\item [(i)]
		The dual SOSC holds at $(\bar y,\bar w,\overline S)$ with respect to $\overline X$, i.e.,
		\begin{equation}\label{dsosc}
			\langle H_w,(\nabla h^*)'(\bar w; H_w) \rangle-\varphi^*_{(\overline S,H_S)}(\overline X)>0,\quad \forall (H_y,H_w,H_S)\in \mathcal C(\bar y,\bar w,\overline S)\backslash\{0\},
		\end{equation}
		where $\mathcal C(\bar y,\bar w,\overline S)$ is the critical cone defined as
		\begin{equation*}
			\mathcal C(\bar y,\bar w,\overline S):=\left\{
			(D_y,D_w,D_S)\in \Re^l\times\Re^d\times\Re^{m\times n}\mid \begin{array}{c}
				\mathcal B^*D_y+\mathcal Q^*D_w+D_S=0, \\
				D_y\in \mathcal C_{\mathcal P^\circ}(\bar y,\mathcal B\overline X-b),\,D_S\in \mathcal C_{\theta^*}(\overline S,\overline X)
			\end{array}
			\right\}.
		\end{equation*}
		\item [(ii)]
		The primal SRCQ holds at $\overline X$ with respect to $(\bar y,\overline S)$, i.e.,
		\begin{equation}\label{psrcq}
			\left(
			\begin{array}{c}
				\mathcal B \\
				\mathcal I_{\Re^{m\times n}} \\
			\end{array}
			\right)\Re^{m\times n}+
			\left(
			\begin{array}{c}
				\mathcal C_{\mathcal P}(\mathcal B\overline X-b,\bar y) \\
				\mathcal C_{\theta}(\overline X,\overline S) \\
			\end{array}
			\right)=\left(
			\begin{array}{c}
				\Re^l \\
				\Re^{m\times n} \\
			\end{array}
			\right).
		\end{equation}
	\end{itemize}
\end{theorem}
\noindent\textbf{Proof.}``(i)$\Longrightarrow$(ii)". Denote
\begin{equation*}
	\Psi:=\left(
	\begin{array}{c}
		\mathcal B \\
		\mathcal I_{\Re^{m\times n}} \\
	\end{array}
	\right)\Re^{m\times n}+
	\left(
	\begin{array}{c}
		\mathcal C_{\mathcal P}(\mathcal B\overline X-b,\bar y) \\
		\mathcal C_{\theta}(\overline X,\overline S) \\
	\end{array}
	\right).
\end{equation*}
Suppose on the contrary that the SRCQ (\ref{psrcq}) does not hold at $\overline X$ for $(\bar y,\overline S)$, i.e., $\Psi\neq\left(
\begin{array}{c}
	\Re^l \\
	\Re^{m\times n} \\
\end{array}
\right)$.
Then ${\rm cl}\Psi\neq\left(
\begin{array}{c}
	\Re^l \\
	\Re^{m\times n} \\
\end{array}
\right)$.
Therefore, there exists $\overline H\in \left(
\begin{array}{c}
	\Re^l \\
	\Re^{m\times n} \\
\end{array}
\right)$
s.t. $\overline H\notin {\rm cl}\Psi$. Since ${\rm cl}\Psi$ is a closed convex cone, $\overline H-\Pi_{{\rm cl}\Psi}(\overline H)=\Pi_{({\rm cl}\Psi)^\circ}(\overline H)\neq0$. Denote $H:=\Pi_{({\rm cl}\Psi)^\circ}(\overline H)=\left(\begin{array}{c}
	H_1\\
	H_2 \\
\end{array}
\right)\in \left(
\begin{array}{c}
	\Re^l \\
	\Re^{m\times n} \\
\end{array}
\right)$.
Obviously, $\langle H, D\rangle\leq 0$ for any $D\in {\rm cl}\Psi$, which implies that
\begin{equation*}
	\mathcal B^*H_1=0,\quad H_1\in \mathcal C_{\mathcal P^\circ}(\bar y,\mathcal B\overline X-b),\quad H_2=0.
\end{equation*}
Thus, we have that $(H_1,0,0)\in \mathcal C(\bar y,\bar w,\overline S)\backslash\{0\}$ and  $\langle 0,(\nabla h^*)'(\bar w; 0) \rangle-\varphi^*_{(\overline S,0)}(\overline X)=0$, which contradicts the
SOSC (\ref{dsosc}) at $(\bar y,\bar w,\overline S)$ with respect to $\overline X$.

``(ii)$\Longrightarrow$(i)"
Suppose on the contrary that the SOSC (\ref{dsosc}) does not hold at $(\bar y,\bar w,\overline S)$ for $\overline X$. Then, $\exists\overline H=(\overline H_y,\overline H_w,\overline H_S)\in \mathcal C(\bar y,\bar w,\overline S)\backslash \{0\}$ such that
\begin{equation*}
	\langle \overline H_w,(\nabla h^*)'(\bar w; \overline H_w) \rangle-\varphi^*_{(\overline S,\overline H_S)}(\overline X)=0.
\end{equation*}
From \cite[Proposition 16]{Ding2017} we know that $\varphi^*_{(\overline S, H)}(\overline X)\leq0$ for any $H\in \Re^{m\times n}$. Moreover, because $h$ is essentially strictly convex, by \cite[Theorem 26.3]{Roc70}, $h^*$ is essentially smooth. Then, $\nabla h^*$ is locally Lipschitz continuous and directionally differentiable on ${\rm int}({\rm dom}\,h^*)$. Since $(\bar y,\bar w,\overline S)$ is the optimal solution,
it is not hard to verify by the mean value theorem that $\langle  H_w,(\nabla h^*)'(\bar w;  H_w) \rangle>0$ for any $H_w\neq 0$ such that $(H_y,H_w,H_S)\in \mathcal C(\bar y,\bar w,\overline S)\backslash\{0\}$. Therefore,
\begin{equation*}
	\overline H_w=0,\quad \varphi^*_{(\overline S,\overline H_S)}(\overline X)=0.
\end{equation*}
	From $\overline H\in \mathcal C(\bar y,\bar w,\overline S)\backslash\{0\}$, we know that
\begin{equation*}
	\mathcal B^*\overline H_y+\overline H_S=0,\quad \overline H_y\in \mathcal C_{\mathcal P^\circ}(\bar y,\mathcal B\overline X-b),\quad \overline H_S\in \mathcal C_{\theta^*}(\overline S,\overline X).
\end{equation*}
Together with (\ref{cripor}) and Proposition \ref{3.2}, we have
\begin{equation*}
	\overline H_y\in (\mathcal C_{\mathcal P}(\mathcal B\overline X-b,\bar y))^\circ,\quad \overline H_S\in(\mathcal C_{\theta}(\overline X,\overline S))^\circ.
\end{equation*}
Because the SRCQ (\ref{psrcq}) holds at $\overline X$ for $(\bar y,\overline S)$, there exist $\widetilde W\in \Re^{m\times n}$ and $\widetilde H=\left(
\begin{array}{c}
	\widetilde H_1 \\
	\widetilde H_2 \\
\end{array}
\right)\in
\left(
\begin{array}{c}
	\mathcal C_{\mathcal P}(\mathcal B\overline X-b) \\
	\mathcal C_{\theta}(\overline X,\overline S) \\
\end{array}
\right)$
such that
\begin{equation*}
	\left(
	\begin{array}{c}
		\mathcal B \\
		\mathcal I_{\Re^{m\times n}} \\
	\end{array}
	\right)\widetilde W+
	\left(
	\begin{array}{c}
		\widetilde H_1\\
		\widetilde H_2\\
	\end{array}
	\right)=\left(
	\begin{array}{c}
		\overline H_y\\
		\overline H_S \\
	\end{array}
	\right),
\end{equation*}
i.e.,
\begin{equation*}
	\mathcal B\widetilde W+\widetilde H_1=\overline H_y,\quad \widetilde W+\widetilde H_2=\overline H_S.
\end{equation*}
Thus,
\begin{equation*}
	\begin{split}
		\langle\overline H,\overline H\rangle&=\langle \overline H_y,\mathcal B\widetilde W+\widetilde H_1\rangle+\langle\overline H_S,\widetilde W+\widetilde H_2\rangle\\
		&=\langle\mathcal B^*\overline H_y+\overline H_S,\widetilde W\rangle+\langle\overline H_y,\widetilde H_1\rangle+\langle\overline H_S,\widetilde H_2\rangle\\
		&=0+\langle\overline H_y,\widetilde H_1\rangle+\langle\overline H_S,\widetilde H_2\rangle\\
		&\leq 0,
	\end{split}
\end{equation*}
which means that $\overline H=0$. This contradicts the hypothesis that $\overline H\neq0$.

Switching the ``primal" and ``dual" in the above Theorem, owing to the special expression of $\mathcal C_{\theta}(P,W)$,
we can also obtain that the primal SOSC is equivalent to the dual SRCQ from Proposition \ref{3.2}. One may refer to \cite[Theorem 5.1]{CS2018} for more details.
\begin{theorem}\label{psosc-}
	Let $\overline X\in \Re^{m\times n}$ be an optimal solution of problem (\ref{p}) with $\mathcal M_P(\overline X)\neq \emptyset$. Let $(\bar y,\overline S)\in \mathcal M_P(\overline X)$. Then the following results are equivalent:
\begin{itemize}
		\item [(i)]
The dual SRCQ holds at $\bar y$ for $\overline X$, i.e.,
	\begin{equation}\label{dsrcq}
		\mathcal Q^*\Re^d+\mathcal B^*\mathcal C_{\mathcal P^\circ}(\bar y,\mathcal B\overline X-b)+\mathcal C_{\theta^*}(\overline S,\overline X)=\Re^{m\times n}.
	\end{equation}
\item[(ii)]
The primal SOSC holds at $\overline X$ for $(\bar y, \overline S)$, i.e.,
	\begin{equation}\label{psosc}
		\langle \mathcal QD,\nabla^2h(\mathcal Q\overline X)\mathcal QD \rangle-\psi^*_{(\overline X,D)}(\overline S)>0,\quad \forall D\in \mathcal C(\overline X)\backslash\{0\},
	\end{equation}
	where $\mathcal C(\overline X):=\{D\in\Re^{m\times n}\mid\mathcal BD\in\mathcal C_{\mathcal P}(\mathcal B\overline X-b,\bar y),\quad D\in\mathcal C_{\theta}(\overline X,\overline S)\}$.
\end{itemize}
\end{theorem}
	In the following, we first give the definition of the KKT mapping for problem (\ref{p}). Then, we demonstrate that the SOSC and the SRCQ are sufficient and necessary for the ICKKTM. Finally, based on the above result, we give a series of characterizations of the robust ICKKTM for problem (\ref{p}).

The KKT conditions (\ref{kktp}) for problem (\ref{p}) can be equivalently expressed as the nonsmooth equation:
\begin{equation*}
	F(X,y,S)=0,
\end{equation*}
where $F: \Re^{m\times n}\times \Re^l\times\Re^{m\times n}\rightarrow \Re^{m\times n}\times \Re^l\times\Re^{m\times n}$ is defined by
\begin{equation*}\label{F}
	F(X,y,S):=\left[
	\begin{array}{c}
		C+\mathcal B^*y+\mathcal Q^*\nabla h(\mathcal QX)+S\\[5pt]
		\mathcal B X-b-\Pi_Q(\mathcal BX-b+y)\\[5pt]
		X-{\rm Prox}_\theta(X+S)
	\end{array}
	\right],\quad
	(X,y,S)\in \Re^{m\times n}\times \Re^l\times\Re^{m\times n}.
\end{equation*}
Denote the KKT mapping for problem (\ref{p}) by
\begin{equation*}\label{kktmapping}
	S_{\rm KKT}(\delta):=\{(X,y,S)\in \Re^{m\times n}\times \Re^l\times \Re^{m\times n}\mid F(X, y, S)=\delta\}.
\end{equation*}
Thus, we know from \cite{DSZ2017} that the isolated calmness of $S_{\rm KKT}$ at the origin for $(\overline X,\bar y,\overline S)$ is equivalent to
$F'((\overline X,\bar y,\overline S);(\Delta Z,\Delta y,\Delta S))=0\Rightarrow(\Delta Z,\Delta y,\Delta S)=0$.

Define the strong SRCQ (SSRCQ) at $\overline X$ for $(\bar y,\overline S)$ as
\begin{equation}\label{psrcq0}
	\left(
	\begin{array}{c}
		\mathcal B \\
		\mathcal I_{\Re^{m\times n}} \\
	\end{array}
	\right)\Re^{m\times n}+
	\left(
	\begin{array}{c}
		\mathcal C_{\mathcal P}(\mathcal B\overline X-b,\bar y) \\
		\mathcal C^0_{\theta}(\overline X,\overline S) \\
	\end{array}
	\right)=\left(
	\begin{array}{c}
		\Re^l \\
		\Re^{m\times n} \\
	\end{array}
	\right).
\end{equation}
Similarly, define the weak SOSC (WSOSC) holds at $\overline X$ for $(\bar y, \overline S)$, i.e.,
\begin{equation}\label{w-sosc}
	\langle \mathcal QD,\nabla^2h(\mathcal Q\overline X)\mathcal QD \rangle-\psi^*_{(\overline X,D)}(\overline S)>0,\quad \forall D\in \mathcal C^0(\overline X)\backslash\{0\},
\end{equation}
where $\mathcal C^0(\overline X):=\{D\in\Re^{m\times n}\mid\mathcal BD\in\mathcal C_{\mathcal P}(\mathcal B\overline X-b,\bar y),\quad D\in\mathcal C^0_{\theta}(\overline X,\overline S)\}$.
\begin{theorem}\label{thiso}
	Let $\overline X\in\Re^{m\times n}$ be an optimal solution of problem (\ref{p}) with $(\bar y,\overline S)\in \mathcal M_P(\overline X)\neq \emptyset$.
	Then the following results are equivalent:
	\begin{itemize}
		\item [(i)] The SSRCQ (\ref{psrcq0}) holds at $\overline X$ for $(\bar y,\overline S)$ and the WSOSC (\ref{w-sosc}) holds at $\overline X$ for $(\bar y,\overline S)$.
		\item [(ii)] The KKT mapping $S_{\rm KKT}$ is isolated calm at the origin for $(\overline X,\bar y, \overline S)$.
	\end{itemize}
\end{theorem}
\begin{proof}
	``(i)$\Longrightarrow$ (ii)".
	Let $(\Delta Z,\Delta y,\Delta S)\in \Re^{m\times n}\times \Re^l\times \Re^{m\times n}$ satisfies
	\begin{equation}\label{F'=0}
		F'((\overline X,\bar y,\overline S);(\Delta Z,\Delta y,\Delta S))=
		\left[
		\begin{array}{c}
			\mathcal Q^*\nabla^2 h(\mathcal Q\overline X)\mathcal Q\Delta Z+\mathcal B^*\Delta y+\Delta S\\[5pt]
			\mathcal B \Delta Z-\Pi'_Q(\mathcal B\overline X-b+\bar y;\mathcal B\Delta Z+\Delta y)\\[5pt]
			\Delta Z-{\rm Prox}'_\theta(\overline X+\overline S;\Delta Z+\Delta S)
		\end{array}
		\right]
		=0.
	\end{equation}
	From the second equation of (\ref{F'=0}) and \cite[Lemma 4.2]{HSZ2015ar} we know that
	\begin{equation}\label{critAx}
		\mathcal B\Delta Z\in \mathcal C_Q(\mathcal B\overline X-b,\bar y),\quad \langle \Delta y,\mathcal B\Delta Z\rangle=0.
	\end{equation}
	Proposition \ref{3.4} and the third equation of (\ref{F'=0}) means that
	\begin{equation}\label{xcirt}
		\Delta Z\in \mathcal C^0_{\theta}(\overline X,\overline S),\quad \langle \Delta Z,\Delta S\rangle=-\psi^*_{(\overline X,\Delta Z)}(\overline S).
	\end{equation}
	Thus, we can conclude from (\ref{critAx}) and (\ref{xcirt}) that
	\begin{equation*}\label{xxx}
		\Delta Z\in \mathcal C^0(\overline X).
	\end{equation*}
	By taking the inner product of $\Delta Z$ and the first equation of (\ref{F'=0}), together with (\ref{critAx}) and (\ref{xcirt}), we can obtain
	\begin{equation*}\label{dx=0}
		\begin{split}
			&\langle \Delta Z, \mathcal Q^*\nabla^2 h(\mathcal Q\overline X)\mathcal Q\Delta Z+\mathcal B^*\Delta y+\Delta S\rangle\\[5pt]
			=&\langle \Delta Z, \mathcal Q^*\nabla^2 h(\mathcal Q\overline X)\mathcal Q\Delta Z \rangle+\langle \Delta Z,\mathcal B^*\Delta y\rangle+\langle \Delta Z,\Delta S\rangle\\[5pt]
			=& \langle \mathcal Q\Delta Z,\nabla^2 h(\mathcal Q\overline X)\mathcal Q\Delta Z \rangle+\langle \mathcal B\Delta Z,\Delta y\rangle+\langle \Delta Z,\Delta S\rangle\\[5pt]
			=&\langle \mathcal Q\Delta Z,\nabla^2 h(\mathcal Q\overline X)\mathcal Q\Delta Z \rangle-\psi^*_{(\overline X,\Delta Z)}(\overline S)\\[5pt]
			=&0.
		\end{split}
	\end{equation*}
	Then, from the WSOSC (\ref{w-sosc}) we know that
	\begin{equation*}\label{x=0}
		\Delta Z=0.
	\end{equation*}
	Therefore, (\ref{F'=0}) reduces to
	\begin{equation}\label{reduF'}
		\left[
		\begin{array}{c}
			\mathcal B^*\Delta y+\Delta S\\[5pt]
			\Pi'_Q(\mathcal B\overline X-b+\bar y;\Delta y)\\[5pt]
			{\rm Prox}'_\theta(\overline X+\overline S;\Delta S)
		\end{array}
		\right]=0.
	\end{equation}
	From Lemma 4.3 in \cite{HSZ2015ar} and Proposition \ref{3.5}, we obtain that (\ref{reduF'}) is equivalent to the following formula
	\begin{equation}\label{polarsrcq}
		\left(
		\begin{array}{c}
			\Delta y \\
			\Delta S \\
		\end{array}
		\right)\in
		\left[\left(
		\begin{array}{c}
			\mathcal B \\
			\mathcal I_{\Re^{m\times n}} \\
		\end{array}
		\right)\Re^{m\times n}+
		\left(
		\begin{array}{c}
			\mathcal C_{\mathcal P}(\mathcal B\overline X-b,\bar y) \\
			\mathcal C^0_{\theta}(\overline X,\overline S) \\
		\end{array}
		\right)\right]^\circ.
	\end{equation}
	According to the SSRCQ (\ref{psrcq0}),
	\begin{equation*}\label{ys=0}
		(\Delta y,\Delta S)=0.
	\end{equation*}
	This, together with $\Delta Z=0$, imply that the KKT mapping is isolated calm at the origin for $(\overline X,\bar y, \overline S)$.
	
	We prove ``(ii)$\Longrightarrow$ (i)" by contradiction. Assume that the SSRCQ (\ref{psrcq0}) does not hold at $\overline X$ for $(\bar y,\overline S)$. Then $\exists(\Delta y,\Delta S)\neq 0$ s.t. (\ref{polarsrcq}) holds, which also means that (\ref{reduF'}) holds. Hence, $F'((\overline X,\bar y,\overline S);(0,\Delta y,\Delta S))=0$. This contradicts with the ICKKTM at the origin for $(\overline X,\bar y, \overline S)$.
	
	Because $\overline X$ is an optimal solution of problem (\ref{p}) and the SSRCQ (\ref{psrcq0}) holds at $\overline X$ for $(\bar y,\overline S)$, the Lagrangian multipliers $\mathcal M_P(\overline X)$ is a singleton and the second order necessary condition holds at $\overline X$, i.e.,
	\begin{equation*}\label{sonc}
		\langle \mathcal QD,\nabla^2h(\mathcal Q\overline X)\mathcal QD \rangle-\psi^*_{(\overline X,D)}(\overline S)\geq0,\quad \forall D\in \mathcal C^0(\overline X).
	\end{equation*}
	Suppose that the WSOSC (\ref{w-sosc}) does not hold at $\overline X$ for $(\bar y,\overline S)$, which means that $\exists\Delta Z'\in \mathcal C^0(\overline X)\backslash\{0\}$ such that
	\begin{equation*}\label{soscnot}
		\langle \mathcal Q \Delta Z',\nabla^2h(\mathcal Q\overline X)\mathcal Q\Delta Z' \rangle-\psi^*_{(\overline X,\Delta Z')}(\overline S)=0.
	\end{equation*}
	By \cite[Proposition 16]{Ding2017}, $\forall H\in \Re^{m\times n}$, $\psi^*_{(\overline X, H)}(\overline S)\leq 0$. Besides, since $h$ is essentially strictly convex,
	$\langle \mathcal Q H,\nabla^2h(\mathcal Q\overline X)\mathcal QH \rangle>0$ for any $H\in \Re^{m\times n}$ satisfies $\mathcal QH\neq 0$. Consequently,
	\begin{equation*}\label{fsigma=0}
		\mathcal Q \Delta Z'=0, \quad \psi^*_{(\overline X,\Delta Z')}(\overline S)=0.
	\end{equation*}
	From $\Delta Z'\in \mathcal C^0(\overline X)\backslash\{0\}$, we know that
	\begin{equation}\label{cx}
		\mathcal B\Delta Z'\in\mathcal C_{\mathcal P}(\mathcal B\overline X-b,\bar y),\quad \Delta Z'\in\mathcal C^0_{\theta}(\overline X,\overline S).
	\end{equation}
	Thus, from \cite[Lemma 4.2]{HSZ2015ar}, the first formula in (\ref{cx}) is equivalent to
	\begin{equation}\label{pi'}
		\mathcal B \Delta Z'=\Pi'_Q(\mathcal B\overline X-b+\bar y;\mathcal B\Delta Z').
	\end{equation}
	We can conclude from the second formula in (\ref{cx}) and (\ref{prox'}) that
	\begin{equation}\label{propx}
		\Delta Z'={\rm Prox}'_\theta(\overline X+\overline S;\Delta Z').
	\end{equation}
	Therefore,  formulas (\ref{pi'}), (\ref{propx}) and $\mathcal Q \Delta Z'=0$ imply that
	\begin{equation*}\label{F'==0}
		F'((\overline X,\bar y,\overline S);(\Delta Z',0,0))=0.
	\end{equation*}
	Notice that $\Delta Z'\neq 0$, which contradicts with the isolated
	calmness of $S_{\rm KKT}$ at the origin for $(\overline X,\bar y, \overline S)$. That is, the WSOSC (\ref{w-sosc}) holds at $\overline X$ for $(\bar y,\overline S)$. This completes the proof.
\end{proof}
It is said that the Robinson constraint qualification (RCQ) holds at $\overline X\in \Re^{m\times n}$ for problem (\ref{p}) if
\begin{equation}\label{rcq}
	\mathcal B\Re^{m\times n}+\mathcal T_{\mathcal P}(\mathcal B\overline X-b)=\Re^l.
\end{equation}
Finally,  we can establish a series of equivalent conditions of the robust ICKKTM for problem (\ref{p}).
\begin{theorem}\label{th1}
	Let $\overline X\in\Re^{m\times n}$and $(\bar y,\bar w,\overline S)\in \Re^l\times\Re^d\times \Re^{m\times n}$ be optimal solutions of problem (\ref{p}) and (\ref{d}), respectively.
	Assume that $(\bar y,\overline S)\in \mathcal M_P(\overline X)$ and $\overline X\in \mathcal M_D(\bar y, \bar w,\overline S)$, and the RCQ (\ref{rcq}) holds at $\overline X$.
	Then the following statements are equivalent:
	\begin{itemize}
		\item [(a)] The KKT mapping $S_{\rm KKT}$ is (robustly) isolated calm at the origin for $(\overline X,\bar y, \overline S)$.
		\item [(b)] The primal WSOSC (\ref{w-sosc}) holds at $\overline X$ for $(\bar y,\overline S)$ and the primal SSRCQ (\ref{psrcq0}) holds at $\overline X$ for $(\bar y,\overline S)$.
		\item [(c)] The primal SOSC (\ref{psosc}) holds at $\overline X$ for $(\bar y,\overline S)$ and the primal SRCQ (\ref{psrcq}) holds at $\overline X$ for $(\bar y,\overline S)$.
		\item[(d)] The primal SOSC (\ref{psosc}) holds at $\overline X$ for $(\bar y,\overline S)$ and the dual SOSC (\ref{dsosc}) holds at $(\bar y,\bar w,\overline S)$ for $\overline X$.
		\item [(e)] The dual SRCQ (\ref{dsrcq}) holds at $(\bar y,\bar w,\overline S)$ for $\overline X$ and the dual SOSC (\ref{dsosc}) holds at $(\bar y,\bar w,\overline S)$ for $\overline X$.
		\item [(f)] The dual SRCQ (\ref{dsrcq}) holds at $(\bar y,\bar w,\overline S)$ for $\overline X$ and the primal SRCQ (\ref{psrcq}) holds at $\overline X$ for $(\bar y,\overline S)$.
	\end{itemize}
\end{theorem}
\begin{proof}
	Since the spectral norm function $\|\cdot\|_2$ is $\mathcal C^2$-cone reducible, the ``robustly" in (i) of the Theorem \ref{th1}, which means that $S_{\rm KKT}$ is also lower semicontinuous at the origin for $(\overline X,\bar y,\overline S)$, can be obtained automatically from \cite[Theorem 17]{DSZ2017}. Besides, since the spectral norm function is Lipschitz continuous at $\overline X$, by combining Theorem \ref{p-srcq}, Theorem \ref{psosc-}, Theorem \ref{thiso}, \cite[Theorem 17]{DSZ2017} and \cite[Proposition 3.3]{CS2018}, we can obtain that $(a)\Leftrightarrow(b)\Leftrightarrow(c)\Leftrightarrow(d)\Leftrightarrow(e)\Leftrightarrow(f)$.
\end{proof}
\begin{remark}
	Proposition 3.3 in \cite{CS2018} is obtained indirectly by transforming problem (\ref{p}) into a matrix conic optimization induced by nuclear norm, and then applying \cite[Theorem 24]{DSZ2017}. In the above theorem, we give a direct proof. The difference is that we need to use the variational properties of the directional derivatives of the proximal mapping of spectral norm, while \cite{DSZ2017} uses the projection mapping.
\end{remark}

\section{Conclusion}

This paper is devoted to establishing a series of equivalent conditions of the robust ICKKTM for spectral norm regularized convex matrix optimization problems. We extend the results for nuclear norm regularized convex optimization problems in \cite{CS2018} to spectral norm regularized convex optimization problems. There is still much work to be done, such as how to apply the variational properties of proximal mapping of spectral norm function to directly prove the equivalence between (a) and (c) in Theorem \ref{th1},  whether the conclusions in this paper or in \cite{CS2018} can be extended to the Ky Fan $k$-norm regularized problems, or whether the variational properties of the spectral norm function obtained in this paper can be applied to provide stability results for the generalized variational inequality problem with the spectral norm function. These are all the topics we will study in the future.

\bibliographystyle{plain}

\begin{thebibliography}{}
	\small
	
	\bibitem{GW1993}  Watson, G.A.:
	 On matrix approximation problems with Ky Fan k norms.
	 Numerical Algorithms.
	 \textbf{5}, 263--272(1993)
	
	\bibitem{TT1998} Toh, K.C., Trefethen, L.N.:
	 The chebyshev polynomials of a matrix.
	 SIAM Journal on Matrix Analysis and Applications.
	\textbf{20}, 400--419(1998)
	
	\bibitem{AD2006} Apkarian, P., Noll,D.:
	 Nonsmooth $H_{\infty}$ Synthesis.
	 IEEE Transactions on Automatic Control.
	 \textbf{51}, 71--86(2006)
	
	\bibitem{VDA2007} Bompart, V., Noll,D., Apkarian,P.:
	 Second-order nonsmooth optimization for $H_{\infty}$ synthesis.
	 Numer. Math.
	 \textbf{107}, 433--454(2007)
	
	\bibitem{Yuan2024} Yuan, J., Yang, D.,  Xun, D., Teo, K., Wu, C., Li, A., et al.:
	Sparse optimal control of cyber-physical systems via PQA approach.
	Pacific Journal of Optimization. In press, (2024)
	
	\bibitem{ANCP2023} Johansson, A., Engsner, N., Stranneg{\aa}rd, C., Mostad, P.:
	Improved Spectral Norm Regularization for Neural Networks.
	In Modeling Decisions for Artificial Intelligence: 20th International Conference, MDAI 2023, Umeå, Sweden, June 19–22, 2023, Proceedings. Springer-Verlag. Berlin. Heidelberg. 181--201(2023)
	
	\bibitem{YM2017} Yoshida, Y., Takeru, M.:
	Spectral Norm Regularization for Improving the Generalizability of Deep Learning. preprint.
	
	\bibitem{GYS2024} Gao, R., Yang, W., Sun, X.:
	Defying Forgetting in Continual Relation Extraction via Batch Spectral Norm Regularization.
	2024 International Joint Conference on Neural Networks (IJCNN), 1--8(2024)
	
	\bibitem{HSZ2018} Han, D.R., Sun, D.F., Zhang, L.W.:
	Linear rate convergence of the alternating direction method of multipliers for convex composite programming.
	Mathematics of Operations Research. \textbf{43}, 622--637(2018)
	
	\bibitem{Roc1976} Rockafellar, R.T.:
	Augmented Lagrangians and applications of the proximal point algorithm in convex programming.
	Mathematics of Operations Research. \textbf{1}, 97--116(1976)
	
	\bibitem{ZZ2016} Zhang, Y.L., Zhang, L.W.:
	On the upper Lipschitz property of the KKT mapping for nonlinear semidefinite optimization.
	Operations Research Letters. \textbf{44}, 474--478(2016)
	
	\bibitem{ZZW2017} Zhang, Y., Zhang, L.W., Wu, J., Wang, K.: Characterizations of local upper Lipschitz property of perturbed solutions to nonlinear second-order cone programs.
	Optimization. \textbf{66}, 1079--1103(2017)
	
	\bibitem{DSZ2017} Ding, C., Sun, D.F., Zhang, L.W.:
	Characterization of the robust isolated calmness for a class of conic programming problems.
	SIAM J. OPTIM. \textbf{5},  67--90(2017)
	
	\bibitem{CS2008}  Chan, Z.X., Sun, D.F.:
	Constraint Nondegeneracy, Strong Regularity, and Nonsingularity in Semidefinite Programming.
	SIAM Journal on Optimization. \textbf{19}, 370--396(2008)
	
	\bibitem{HSZ2015ar}  Han, D.R., Sun, D.F., Zhang, L.W.:
	Linear rate convergence of the alternating direction method of multipliers for convex composite quadratic and semi-definite programming. arXiv:1508.02134 (2015).
	
	\bibitem{CS2018} Cui, Y., Sun, D.F.:
	A complete characterization of the robust isolated calmness of nuclear norm regularized convex optimization problems.
    Journal of Computational Mathematics. \textbf{36},  441--458(2018)
	
	\bibitem{Row98} Rockafellar, R.T., Wets R.J.-B.:
	Variational Analysis.
	Springer-Verlag, New York (1998)
	
	\bibitem{BS00} Bonnans, J.F., Shapiro, A.:
	Perturbation Analysis of Optimization Problems.
	Springer, New York (2000)
	
	\bibitem{Fan1949} Fan, K.:
	On a theorem of Weyl concerning eigenvalues of linear transformations. Proceedings of the National Academy of Sciences of the United States of America. \textbf{35}, 652--655(1949)
	
	\bibitem{Roc70}
	Rockafellar, R.T.:
	Convex Analysis.
	Princeton University Press, Princeton, New Jersey (1970)
	
	\bibitem{beck2017} Beck, A.:
	First-Order Methods in Optimization.
	Society for Industrial and Applied Mathematics and the Mathematical Optimization Society. Philadelphia, (2017)
	
	\bibitem{Watson1992} Watson, G.A.:
	Characterization of the subdifferential of some matrix norms.
	Linear Algebra Appl. \textbf{170},  33--45(1992)
	
	\bibitem{DSST2018} Ding, C., Sun, D.F., Sun, J., Toh, K.C.:
	Spectral operators of matrices.  Math. Program. \textbf{168}, 509--531(2018)
	
	\bibitem{ZZX2013} Zhang, L.W., Zhang, N., Xiao, X.T.:
	On the second-order directional derivatives of singular values of matrices and symmetric matrix-valued functions. Set-Valued Var. \textbf{21}, 557--586(2013)
	
	
	\bibitem{Ding2017} Ding, C.:
	Variational analysis of the Ky Fan $k$-norm. Set-Valued Var. Anal. \textbf{25}, 265--296(2017)
	
\end{thebibliography}

\end{document}